\documentclass{article}
\usepackage{amsfonts}
\usepackage{amsmath}
\usepackage{amssymb}
\usepackage{amsthm}
\usepackage[margin=1.3in]{geometry}
\usepackage{hyperref}
\newtheorem{thm}{Theorem}
\newtheorem{lemma}{Lemma}

\newtheorem{conj}{Conjecture}

\newcommand{\red}{\text{red}}
\newcommand{\weight}{\text{weight}}

\begin{document}

\title{Approaches for enumerating permutations with a prescribed number of occurrences of patterns}

\author{Brian Nakamura\thanks{Mathematics Department, Rutgers University-New Brunswick, Piscataway, NJ, USA. [bnaka@math.rutgers.edu]}}

\date{}

\maketitle

\begin{abstract}
	In recent work, Zeilberger and the author used a functional equations approach for enumerating permutations with $r$ occurrences of the pattern $1 2 \ldots k$. In particular, the approach yielded a polynomial-time enumeration algorithm for any fixed $r \geq 0$. We extend that approach to patterns of the form $1 2 \ldots (k-2) (k) (k-1)$ by deriving analogous functional equations and using them to develop similar algorithms that enumerate permutations with $r$ occurrences of the pattern. We also generalize those techniques to handle patterns of the form $2 3 \ldots k 1$ and derive analogous functional equations and enumeration algorithms. Finally, we show how the functional equations and algorithms can be modified to track inversions as well as handle multiple patterns simultaneously. This paper is accompanied by Maple packages that implement the algorithms described.
\end{abstract}

\section{Introduction}

Let $\sigma = \sigma_{1} \ldots \sigma_{k}$ be a sequence of $k$ distinct positive integers. We define the \emph{reduction} $\text{red}(\sigma)$ to be the length $k$ permutation $\tau = \tau_{1} \ldots \tau_{k}$ that is order-isomorphic to $\sigma$ (i.e., $\sigma_{i} < \sigma_{j}$ if and only if $\tau_{i} < \tau_{j}$ for every $i$ and $j$). Given a (permutation) pattern $\tau \in \mathcal{S}_{k}$, we say that a permutation $\pi = \pi_{1} \ldots \pi_{n}$ \emph{contains} the pattern $\tau$ if there exists $1 \leq i_{1} < i_{2} < \ldots < i_{k} \leq n$ such that $\text{red}(\pi_{i_{1}} \pi_{i_{2}} \ldots \pi_{i_{k}}) = \tau$, in which case we call $\pi_{i_{1}} \pi_{i_{2}} \ldots \pi_{i_{k}}$ an \emph{occurrence} of $\tau$. We will define $N_{\tau}(\pi)$ to be the number of occurrences of $\tau$ in $\pi$. For example, if the pattern $\tau = 1 2 3$, the permutation $5 3 4 1 2$ avoids the pattern $\tau$ (so $N_{1 2 3}(5 3 4 1 2) = 0$), whereas the permutation $5 2 1 3 4$ contains two occurrences of $\tau$ (so $N_{1 2 3}(5 2 1 3 4) = 2$).\\

For a pattern $\tau$ and non-negative integer $r \geq 0$, we define the set
\begin{align*}
	\mathcal{S}_{n}(\tau,r) := \left\{ \pi \in \mathcal{S}_{n} : \pi \text{ has exactly } r \text{ occurrences of the pattern } \tau \right\}
\end{align*}
and also define $s_{n}(\tau,r) := \left| \mathcal{S}_{n}(\tau,r) \right|$. The corresponding generating function is defined as
\begin{align*}
	F_{\tau}^{r}(x) := \mathop{\sum} \limits_{n=0}^{\infty} {s_{n}(\tau,r) x^{n}}.
\end{align*}
Observe that the classical pattern avoidance problem corresponds to the case where $r=0$ and has been well studied. In this setting, $\mathcal{S}_{n}(\tau,0)$ is known to be enumerated by the Catalan numbers if $\tau \in \mathcal{S}_{3}$ \cite{knuth}. However, much is still unknown even for length $4$ patterns. For example, permutations avoiding the pattern $1 3 2 4$ have been notoriously difficult to enumerate. Precise asymptotics are not even known, although B{\'o}na recently gave an improved upper bound for the growth rate (in \cite{bona:ub1324}) by modifying the approach used by Claesson, Jel{\'{\i}}nek, and Steingr{\'{\i}}msson~\cite{CJS:ub1324}. The survey paper by Kitaev and Mansour~\cite{kitman:survey} provides an extensive overview of work in this area as well as related problems in permutation patterns.\\

While the more general problem (where $r \geq 0$) has also been studied, the work has usually been restricted to small patterns (usually length three) and small $r$. In~\cite{noonan}, Noonan studied permutations containing exactly one occurrence of $1 2 3$ and proved that $s_{n}(1 2 3, 1) = \frac{3}{n} {2 n \choose n-3}$. Burstein recently gave a short combinatorial proof for the result~\cite{bur:f321}. In~\cite{noonzeil:forbid}, Noonan and Zeilberger presented an approach using functional equations to enumerate $s_{n}(\tau, r)$ for small $r$ and for the patterns $1 2 3$, $3 1 2$, and $1 2 3 4$. Subsequent work has been done by B{\'o}na~\cite{bona:f132, bona:prec132}, Fulmek~\cite{fulmek}, Mansour and Vainshtein~\cite{manvain}, Callan~\cite{callan}, and many others. Many of these focused on finding $F_{\tau}^{r}(x)$ for $\tau \in \mathcal{S}_{3}$ and for small $r$.\\

One difficulty arising from the initial Noonan-Zeilberger functional equation approach in~\cite{noonzeil:forbid} was that the approach became very complicated for even $r=2$. In addition, there are many patterns that this approach does not readily extend to. One such pattern (explicitly mentioned in~\cite{noonzeil:forbid}) was $1 4 3 2$. A modified approach was recently presented in~\cite{nz:gwilf} for the case of increasing patterns. Given a fixed $r \geq 0$, the resulting enumeration algorithm for computing $s_{n}(1 2 \ldots k,r)$ was polynomial-time (in $n$). This, in a sense, tackled the first difficulty from~\cite{noonzeil:forbid} and allowed us to enumerate the sequence $s_{n}(1 2 \ldots k, r)$ for even larger fixed $r$.\\

In this paper, we extend the enumeration techniques in~\cite{nz:gwilf} to new families of patterns (including the pattern $1 4 3 2$) as well as multiple patterns. It should be noted that this general approach is different from the \emph{enumeration schemes} approach pioneered by Zeilberger~\cite{zeil:wilf} and extended by Vatter~\cite{vv:wilfp}, Pudwell~\cite{pud:barsch, baxpud:vincsch}, and Baxter~\cite{baxter:inv, baxpud:vincsch}. The enumeration schemes approach is useful for enumerating pattern-avoiding permutations (the $r=0$ case) but does not appear to be readily adaptable to the generalized setting for permutations with $r>0$ occurrences of a pattern.\\

The paper is organized in the following manner. Section~\ref{incrT} extends the approach in~\cite{nz:gwilf} to the patterns $1 3 2$, $1 2 4 3$, $1 2 3 5 4$, and so on. Section~\ref{p2k1} generalizes the techniques used in~\cite{nz:gwilf} and applies them to the patterns $2 3 1$, $2 3 4 1$ (which is equivalent to $1 4 3 2$), and so on. Section~\ref{multipatt} extends this approach to handle multiple patterns simultaneously as well as refining by the inversion number\footnote{This is technically the same as tracking the number of $2 1$ patterns that occur.}. Section~\ref{concl} lists some possible future work as well as some conjectures. The enumeration algorithms developed in this paper are implemented in the Maple packages {\tt FINCR}, {\tt FINCRT}, {\tt F231}, {\tt F2341}, {\tt F123n132}, {\tt F1234n1243}, and {\tt FS3}. They are all available from the author's website.

\section{Counting occurrences of the pattern $1 2 \ldots (k-2) (k) (k-1)$}\label{incrT}

In this section, we adapt the approach in~\cite{nz:gwilf} (for increasing patterns) to the patterns $1 2 \ldots (k-2) (k) (k-1)$. We first handle the case of $1 3 2$ in full detail and then outline how to generalize this approach to patterns $1 2 4 3$, $1 2 3 5 4$, and so on.

\subsection{Permutations containing $1 3 2$}

Given a (fixed) pattern $\tau$ and non-negative integer $n$, we define the polynomial
\begin{align}
	f_{n}(t) &:= \mathop{\sum} \limits_{\pi \in \mathcal{S}_{n}} {t^{N_{\tau}(\pi)}}.
\end{align}
Observe that the coefficient of $t^{r}$ in $f_{n}(t)$ is exactly equal to $s_{n}(\tau,r)$. For a fixed pattern $\tau$ and fixed $r \geq 0$, our goal is to quickly compute $s_{n}(\tau,r)$. In the remainder of this section, we will assume that $\tau = 1 3 2$.\\

In addition to the variable $t$, we introduce the catalytic variables $x_{1}, \ldots, x_{n}$ and define the weight of a length $n$ permutation $\pi = \pi_{1} \ldots \pi_{n}$ to be
\begin{align*}
	\text{weight}_{132}(\pi) := t^{N_{132}(\pi)} \mathop{\prod} \limits_{i=1}^{n} {x_{i}^{\# \{ (a,b) \; : \; \pi_{a} > \pi_{b} = i,\; 1 \leq a < b \leq n\}}}.
\end{align*}
In general, this will be written more simply as $\text{weight}(\pi)$ when the fixed pattern is clear from context (in this case $1 3 2$). For example, $\text{weight}(1 2 3 4 5) = 1$, $\text{weight}(1 3 2 4 5) = t x_{2}$, and $\text{weight}(2 5 1 4 3) = t^{4} x_{1}^{2} x_{3}^{2} x_{4}$. In essence, the weight of a permutation encodes the number of $1 3 2$ patterns as well as information on $2 1$ patterns (which may become the ``$3 2$'' of a $1 3 2$ pattern if a term is inserted at the beginning of the permutation).\\

For each $n$, we define the polynomial
\begin{align*}
	P_{n}(t; x_{1}, \ldots, x_{n}) := \mathop{\sum} \limits_{\pi \in \mathcal{S}_{n}} {\text{weight}(\pi)}.
\end{align*}
Observe that $P_{n}$ is essentially a generalized multi-variate polynomial for $f_{n}$ and in particular, $P_{n}(t; 1, \ldots, 1) = f_{n}(t)$. We now get the following:\\

\begin{lemma}
	\label{lemma132}
	Let $\pi = \pi_{1} \ldots \pi_{n}$ and suppose that $\pi_{1} = i$. If $\pi^{\prime} := \red(\pi_{2} \ldots \pi_{n})$, then
	\begin{align*}
		\weight(\pi) = x_{1} x_{2} \ldots x_{i-1} \cdot \weight(\pi^{\prime}) \vert_{x_{i} \rightarrow t x_{i+1} \; , \; x_{i+1} \rightarrow t x_{i+2} \; , \; \ldots \; , \; x_{n-1} \rightarrow t x_{n}}. 
	\end{align*}
\end{lemma}

\begin{proof}
	We assume $i$ to be a fixed value and will compute $\weight(\pi)$ from $\weight(\pi^{\prime})$. We re-insert $i$ at the beginning of $\pi^{\prime}$ by shifting all the terms $i, i+1, \ldots , n-1$ up by $1$ (i.e., $x_{j} \rightarrow x_{j+1}$ for $j \geq i$). The new ``$i$'' would create new $2 1$ patterns and would require an extra factor of $x_{1} x_{2} \ldots x_{i-1}$ for the weight. Also, observe that $N_{1 3 2}(\pi)$ is equal to the number of occurrences of $1 3 2$ in $\pi^{\prime}$ \textbf{plus} the number of occurrences of $2 1$ in $\pi_{2} \ldots \pi_{n}$, where the term corresponding to the ``$1$'' is larger than $i$. Therefore, our $x_{j}$ shift now becomes $x_{j} \rightarrow t x_{j+1}$ for $j \geq i$.
\end{proof}

This directly leads to the functional equation:\\

\begin{thm}
	For the pattern $\tau = 1 3 2$, 
	\begin{align}
		P_{n}(t; x_{1}, \ldots, x_{n}) = \mathop{\sum} \limits_{i=1}^{n} {x_{1} x_{2} \ldots x_{i-1} \cdot P_{n-1}(t; x_{1}, \ldots, x_{i-1}, t x_{i+1}, \ldots, t x_{n})}. \tag{FE132} \label{FE132}
	\end{align}
\end{thm}


Once $P_{n}(t; x_{1}, \ldots, x_{n})$ is computed, the catalytic variables $x_{1}, \ldots, x_{n}$ can all be set to $1$ to get $f_{n}(t) = P_{n}(t; 1, \ldots, 1)$. However, it is not necessary to compute $P_{n}(t; x_{1}, \ldots, x_{n})$ in its entirety prior to setting the catalytic variables to $1$. Observe that by (\ref{FE132}), we have:
\begin{align*}
	P_{n}(t; 1, \ldots, 1) = \mathop{\sum} \limits_{i=1}^{n} {P_{n-1}(t; 1 \; [i-1 \; \; times], t \; [n-i \; \; times])}.
\end{align*}
We get terms of the form $P_{a_{0}+a_{1}}(t; 1 \; [a_{0} \; times], t \; [a_{1} \; times])$ in the summation, which can again be plugged into (\ref{FE132}) to get:
\begin{align*}
	P_{a_{0}+a_{1}}(t; 1 \; [a_{0} \; times], t \; [a_{1} \; times]) = \mathop{\sum} \limits_{i=1}^{a_{0}} {P_{a_{0}+a_{1}-1}(1 \; \left[i-1 \; \; times \right], t \; \left[a_{0}-i \; \; times \right], t^{2} \; \left[a_{1} \; \; times \right])}\\
	+ \mathop{\sum} \limits_{i=1}^{a_{1}} {t^{i-1} P_{a_{0}+a_{1}-1}(1 \; \left[a_{0} \; \; times \right], t \; \left[i-1 \; \; times \right], t^{2} \; \left[a_{1}-i \; \; times \right])}
\end{align*}
Now, we must deal with terms of the form $P_{a_{0}+a_{1}+a_{2}}(t; 1 \; [a_{0} \; times], t \; [a_{1} \; times], t^{2} \; [a_{2} \; times])$. We can continue this recursive process of plugging new terms into (\ref{FE132}) to eventually compute $f_{n}(t) = P_{n}(t; 1 \; [n \; \; times])$. This is much faster than the direct weighted counting of all $n!$ permutations, although it is still unfortunately an exponential-time (and memory) algorithm.\\

This algorithm has been implemented in the procedure {\tt F132full(n,t)} (in the Maple package {\tt FINCRT}). For example, the Maple call {\tt F132full(8,t);} computes $f_{8}(t)$ and outputs:
\begin{gather*}
	{t}^{31}+7\,{t}^{30}+20\,{t}^{28}+37\,{t}^{27}+41\,{t}^{26}+109\,{t}^{25}+162\,{t}^{24}+169\,{t}^{23}+322\,{t}^{22}+397\,{t}^{21}+647\,{t}^{20}+730\,{t}^{19}\\
	+1048\,{t}^{18}+1152\,{t}^{17}+1417\,{t}^{16}+1576\,{t}^{15}+1770\,{t}^{14}+1853\,{t}^{13}+2321\,{t}^{12}+2088\,{t}^{11}+2620\,{t}^{10}\\
	+2401\,{t}^{9}+2682\,{t}^{8}+2489\,{t}^{7}+2858\,{t}^{6}+2225\,{t}^{5}+2593\,{t}^{4}+1918\,{t}^{3}+1950\,{t}^{2}+1287\,t+1430
\end{gather*}

Suppose that for a small fixed $r \geq 0$, we wanted the first $20$ terms of the sequence $s_{n}(1 3 2, r)$. By this functional equation approach, one would compute $f_{n}(t)$ and extract the coefficient of $t^{r}$ for each $n$ up to $20$. This approach would expend quite a bit of computational effort in generating unnecessary information (namely, all the $t^{k}$ terms where $k > r$). This issue can mostly be circumvented, however, by a couple of observations. The first is the following lemma from~\cite{nz:gwilf}:
\begin{lemma}
\label{crucial}
	Let $n = a_{0} + a_{1} + \ldots + a_{s}$ (where $a_{i} \geq 0$ for each $i$) and suppose $s > r + 1$. Then, the coefficients of $t^{0}, t^{1}, \ldots, t^{r}$ in
	\begin{gather*}
		P_{n}(t; 1 \left[a_{0} \text{ times} \right], \ldots, t^{s-1} \left[a_{s-1} \text{ times} \right], t^{s} \left[a_{s} \text{ times} \right])\\
		- P_{n}(t; 1 \left[a_{0} \text{ times} \right], \ldots, t^{r} \left[a_{r} \text{ times} \right],  t^{r+1} \left[a_{r+1}+a_{r+2}+\ldots+a_{s} \text{ times} \right])
	\end{gather*}
	all vanish.
\end{lemma}

\begin{proof}
	The more general function $P_{n}(t; x_{1}, \ldots, x_{n})$ is a multi-variate polynomial.
\end{proof}
\noindent This lemma allows us to collapse all the higher powers of $t$ into the $t^{r+1}$ coefficient and allows us to consider objects of the form $P_{n}(t; 1 \left[a_{0} \text{ times} \right], \ldots, t^{r} \left[a_{r} \text{ times} \right],  t^{r+1} \left[a_{r+1} \text{ times} \right])$ regardless of how large $n$ is.\\

Let $n := a_{0} + a_{1} + \ldots + a_{r+1}$. Also, for any expression $R$ and positive integer $k$, let $R\$k$ denote $R \left[k \; \; times \right]$. For example, $t^{3}\$4$ is shorthand for $t^{3},t^{3},t^{3},t^{3}$. Now for any polynomial $p(t)$ in the variable $t$, let $p^{(r)}(t)$ denote the polynomial of degree (at most) $r$ obtained by discarding all powers of $t$ larger than $r$. Also, define the operator $\text{CHOP}_{r}$ by $\text{CHOP}_{r}[p(t)] := p^{(r)}(t)$. \\

An application of (\ref{FE132}) and $\text{CHOP}_{r}$ to $P_{n}^{(r)}(t; 1 \left[a_{0} \text{ times} \right], \ldots, t^{r} \left[a_{r} \text{ times} \right],  t^{r+1} \left[a_{r+1} \text{ times} \right])$ becomes:

\begin{gather*}
	P_{n}^{(r)}(t; 1\$a_{0}, \ldots, t^{r}\$a_{r},  t^{r+1}\$a_{r+1})\\
	= \text{CHOP}_{r} \Biggl[ \mathop{\sum} \limits_{i=1}^{a_{0}} {P_{n-1}^{(r)}(t; 1\$(i-1), t\$(a_{0}-i), t^{2}\$a_{1}, \ldots, t^{r}\$a_{r-1}, t^{r+1}\$(a_{r}+a_{r+1}))}\\
	+ \mathop{\sum} \limits_{i=1}^{a_{1}} {t^{i-1} P_{n-1}^{(r)}(t; 1\$a_{0}, t\$(i-1), t^{2}\$(a_{1}-i), t^{3}\$a_{2}, \ldots, t^{r}\$a_{r-1}, t^{r+1}\$(a_{r}+a_{r+1}))}\\
	+ \mathop{\sum} \limits_{i=1}^{a_{2}} {t^{a_{1} + 2 (i-1)} P_{n-1}^{(r)}(t; 1\$a_{0}, t\$a_{1}, t^{2}\$(i-1), t^{3}\$(a_{2}-i), \ldots, t^{r}\$a_{r-1}, t^{r+1}\$(a_{r}+a_{r+1}))}\\
	+ \ldots\ldots\\
	+ \mathop{\sum} \limits_{i=1}^{a_{r+1}} {t^{a_{1} + 2 a_{2} + \ldots + r a_{r} + (r+1)(i-1)} P_{n-1}^{(r)}(t; 1\$a_{0}, t\$a_{1}, \ldots, t^{r}\$a_{r}, t^{r+1}\$(a_{r+1}-1))} \Biggr].
\end{gather*}

Due to the $\text{CHOP}_{r}$ operator, many terms automatically disappear because of the power of $t$ in front. From a computational perspective, this observation eliminates many unnecessary terms and hence circumvents a lot of unnecessary computation. This has been automated in the Maple package {\tt FINCRT} so that a computer can derive a ``scheme'' for any fixed $r$ (completely on its own) and use it to enumerate $s_{n}(1 3 2, r)$ for as many terms as the user wants.\footnote{The ``scheme'' mentioned here is a liberal application of the word and differs from \emph{enumeration schemes}.}\\

For example, the Maple call {\tt F132rN(5,15);} for the first $15$ terms of $s_{n}(1 3 2, 5)$ produces the sequence:
\begin{gather*}
	0, 0, 0, 0, 5, 55, 394, 2225, 11539, 57064, 273612, 1283621, 5924924, 27005978, 121861262
\end{gather*}


\subsection{Extending to the pattern $1 2 4 3$}

First, we outline how to extend the previous approach to the pattern $1 2 4 3$. In addition to the variable $t$, we now introduce $2 n$ catalytic variables $x_{1}, \ldots, x_{n}$ and $y_{1}, \ldots, y_{n}$. The weight of a length $n$ permutation $\pi = \pi_{1} \ldots \pi_{n}$ will now be
\begin{align*}
	\text{weight}(\pi) := t^{N_{1243}(\pi)} \mathop{\prod} \limits_{i=1}^{n} {x_{i}^{\# \{ (a,b) \; : \; \pi_{a} > \pi_{b} = i,\; 1 \leq a < b \leq n\}} \cdot y_{i}^{\# \{ (a,b,c) \; : \; \pi_{a} = i < \pi_{c} < \pi_{b},\; 1 \leq a < b < c \leq n\}} }.
\end{align*}
For example, $\text{weight}(1 2 3 4 5 6) = 1$ and $\text{weight}(1 3 5 6 2 4) = t^{2} x_{2}^{3} x_{4}^{2} y_{1}^{5} y_{3}^{2}$. In essence, the weight encodes the number of $1 2 4 3$ patterns as well as information on $1 3 2$ and $2 1$ patterns (which may become the ``$2 4 3$'' or ``$4 3$'' of a $1 2 4 3$ when terms are inserted at the beginning of the permutation).\\

For each $n$, we define the polynomial
\begin{align*}
	P_{n}(t; x_{1}, \ldots, x_{n}; \; y_{1}, \ldots, y_{n}) := \mathop{\sum} \limits_{\pi \in \mathcal{S}_{n}} {\text{weight}(\pi)}.
\end{align*}
We now observe the following:\\

\begin{lemma}
	\label{lemma1243}
	Let $\pi = \pi_{1} \ldots \pi_{n}$ and suppose that $\pi_{1} = i$. If $\pi^{\prime} := \red(\pi_{2} \ldots \pi_{n})$, then
	\begin{align*}
		\weight(\pi) = x_{1} x_{2} \ldots x_{i-1} \cdot \weight(\pi^{\prime}) \vert_{A} \quad,
	\end{align*}	
	where $A$ is the set of substitutions given by
	\begin{align*}
		A := \begin{cases}
			x_{b} \rightarrow y_{i} x_{b+1} & i \leq b \leq n-1\\
			y_{b} \rightarrow t y_{b+1} & i \leq b \leq n-1 \quad .
		\end{cases}
	\end{align*}
		
\end{lemma}

\begin{proof}
	We assume $i$ to be a fixed value and will again compute $\weight(\pi)$ from $\weight(\pi^{\prime})$. First, we re-insert $i$ at the beginning of $\pi^{\prime}$ by shifting all the terms $i, i+1, \ldots , n-1$ up by $1$ (i.e., $x_{j} \rightarrow x_{j+1}$ and $y_{j} \rightarrow y_{j+1}$ for $j \geq i$). The new ``$i$'' would create new $2 1$ patterns and would require an extra factor of $x_{1} x_{2} \ldots x_{i-1}$ for the weight. Also, the re-insertion of $i$ would create new $1 3 2$ patterns. The number of such new patterns is exactly the number of $2 1$ patterns in the shifted $\pi^{\prime}$, where the ``$1$'' is greater than $i$. Therefore, our $x_{j}$ shift now becomes $x_{j} \rightarrow y_{i} x_{j+1}$ for $j \geq i$. Also, observe that $N_{1 2 4 3}(\pi)$ is equal to the number of occurrences of $1 2 4 3$ in $\pi^{\prime}$ \textbf{plus} the number of occurrences of $1 3 2$ in $\pi_{2} \ldots \pi_{n}$, where the term corresponding to the ``$1$'' is larger than $i$. Therefore, our $y_{j}$ shift now becomes $y_{j} \rightarrow t y_{j+1}$ for $j \geq i$.
\end{proof}

This directly leads to the new functional equation:\\

\begin{thm}
	For the pattern $\tau = 1 2 4 3$, 
	\begin{gather}
		P_{n}(t; x_{1}, \ldots, x_{n}; \; y_{1}, \ldots, y_{n}) =\notag \\
	\mathop{\sum} \limits_{i=1}^{n} {x_{1} x_{2} \ldots x_{i-1} \cdot P_{n-1}(t; \; x_{1}, \ldots, x_{i-1}, y_{i} x_{i+1}, \ldots, y_{i} x_{n}; \; y_{1}, \ldots, y_{i-1}, t y_{i+1}, \ldots, t y_{n})}. \tag{FE1243} \label{FE1243}
	\end{gather}
\end{thm}


Again, our goal is to compute $f_{n}(t) = P_{n}(t; 1 \left[2 n \text{ times} \right])$. We can apply the same computational methods as before. For example, we can apply (\ref{FE1243}) directly to  $P_{n}(t; 1 \left[2 n \text{ times} \right])$ (and more generally, to objects of the form $P_{n}(t; 1 \left[a_{0} \text{ times} \right], \ldots, t^{s_{1}} \left[a_{s_{1}} \text{ times} \right], 1 \left[b_{0} \text{ times} \right], \ldots, t^{s_{2}} \left[b_{s_{2}} \text{ times} \right])$) to compute $f_{n}(t)$. This again gives us an algorithm that is faster than the direct weighted counting of $n!$ permutations but is still exponential-time (and memory).\\

This algorithm has been implemented in the procedure {\tt F1243full(n,t)} (in Maple package {\tt FINCRT}). For example, the Maple call {\tt F1243full(8,t);} computes $f_{8}(t)$ and outputs:
\begin{gather*}
	{t}^{36}+{t}^{31}+10\,{t}^{30}+3\,{t}^{28}+13\,{t}^{27}+9\,{t}^{26}+8\,{t}^{25}+37\,{t}^{24}+16\,{t}^{23}+16\,{t}^{22}+49\,{t}^{21}+60\,{t}^{20}\\
	+41\,{t}^{19}+130\,{t}^{18}+81\,{t}^{17}+157\,{t}^{16}+266\,{t}^{15}+184\,{t}^{14}+233\,{t}^{13}+542\,{t}^{12}+356\,{t}^{11}+771\,{t}^{
10}\\
	+877\,{t}^{9}+975\,{t}^{8}+972\,{t}^{7}+2180\,{t}^{6}+1710\,{t}^{5}+2658\,{t}^{4}+3119\,{t}^{3}+4600\,{t}^{2}+4478\,t+15767
\end{gather*}

Additionally, both the obvious analog of Lemma~\ref{crucial} as well as the computational reduction using the $\text{CHOP}_{r}$ operator still apply in this setting. This has also been automated in the Maple package {\tt FINCRT}.\\

For example, the Maple call {\tt F1243rN(1,15);} for the first $15$ terms of $s_{n}(1 2 4 3, 1)$ produces the sequence:
\begin{gather*}
	0, 0, 0, 1, 11, 88, 638, 4478, 31199, 218033, 1535207, 10910759, 78310579, 567588264, 4152765025
\end{gather*}
and the Maple call {\tt F1243rN(2,15);} for the first $15$ terms of $s_{n}(1 2 4 3, 2)$ produces the sequence:
\begin{gather*}
	0, 0, 0, 0, 4, 56, 543, 4600, 36691, 284370, 2174352, 16533360, 125572259, 955035260, 7283925999
\end{gather*}


\subsection{Extending to longer patterns}

The approach for the patterns $1 3 2$ and $1 2 4 3$ can be extended analogously to longer patterns of the form $1 2 \ldots (k-2) (k) (k-1)$. For example, if the pattern $\tau = 1 2 3 5 4$, we consider the variable $t$ and $3 n$ catalytic variables: $x_{1}, \ldots, x_{n}$ and $y_{1}, \ldots, y_{n}$ and $z_{1}, \ldots, z_{n}$. The weight of a length $n$ permutation $\pi = \pi_{1} \ldots \pi_{n}$ will now be
\begin{gather*}
	\text{weight}(\pi) = t^{N_{12354}(\pi)} \mathop{\prod} \limits_{i=1}^{n} {x_{i}^{\# \{ (a,b) \; : \; \pi_{a} > \pi_{b} = i\}} \cdot y_{i}^{\# \{ (a,b,c) \; : \; \pi_{a} = i < \pi_{c} < \pi_{b}\}} \cdot z_{i}^{\# \{ (a,b,c,d) \; : \; \pi_{a} = i < \pi_{b}< \pi_{d} < \pi_{c}\}} }
\end{gather*}
where it is always assumed that $a < b < c < d$.\\

An analogous functional equation is derived for the corresponding polynomial 
\begin{align*}
	P_{n}(t; x_{1}, \ldots, x_{n}; \; y_{1}, \ldots, y_{n}; \; z_{1}, \ldots, z_{n}) := \mathop{\sum} \limits_{\pi \in \mathcal{S}_{n}} {\text{weight}(\pi)}
\end{align*}
and all the analogous computational methods work in this setting as well. The $1 2 3 5 4$ case has also been automated in the Maple package {\tt FINCRT}.\\

For example, the Maple call {\tt F12354rN(0,14);} for the first $14$ terms of $s_{n}(1 2 3 5 4, 0)$ produces the sequence:
\begin{gather*}
	1, 2, 6, 24, 119, 694, 4582, 33324, 261808, 2190688, 19318688, 178108704, 1705985883, 16891621166
\end{gather*}
and the Maple call {\tt F12354rN(1,15);} for the first $15$ terms of $s_{n}(1 2 3 5 4, 1)$ produces the sequence:
\begin{gather*}
	0, 0, 0, 0, 1, 19, 246, 2767, 29384, 305646, 3170684, 33104118, 349462727, 3738073247, 40549242195
\end{gather*}


\section{Counting occurrences of the pattern $2 3 \ldots k 1$}\label{p2k1}

In this section, we extend the previous techniques to handle patterns of the form $2 3 \ldots k 1$. Although $s_{n}(2 3 1,r) = s_{n}(1 3 2,r)$ for every $r$ and $n$ (by reversal), we will develop an approach for handling $2 3 1$ directly\footnote{As opposed to computing the equivalent pattern $1 3 2$.} and then show how this can be extended to longer patterns of the form $2 3 \ldots k 1$. This new approach for handling $2 3 1$ will also be necessary in the next section for considering multiple patterns simultaneously.

\subsection{Permutations containing $2 3 1$}
In this section, we will assume that our (fixed) pattern $\tau = 2 3 1$. We define the analogous polynomial
\begin{align*}
	f_{n}(t) := \mathop{\sum} \limits_{\pi \in \mathcal{S}_{n}} t^{N_{2 3 1}(\pi)}.
\end{align*}
Recall that the coefficient of $t^{r}$ in $f_{n}(t)$ will be exactly $s_{n}(2 3 1,r)$.\\

In addition to the variable $t$, we introduce $n(n+1)/2$ catalytic variables $x_{i,j}$ with $1 \leq j \leq i \leq n$ and define the weight of a permutation $\pi = \pi_{1} \ldots \pi_{n}$ to be
\begin{align*}
	\text{weight}(\pi) := t^{N_{2 3 1}(\pi)} \mathop{\prod} \limits _{1 \leq j \leq i \leq n} x_{i,j}^{\# \{ (a,b) \; : \; \pi_{a} > \pi_{b} , \; \pi_{a}=i, \; \pi_{b}<j, \; 1 \leq a < b \leq n\}}
\end{align*}
For example, $\text{weight}(1 2 3 4 5) = 1$ and $\text{weight}(2 4 1 5 3) = t^{2} x_{2,2} x_{4,2} x_{4,3} x_{4,4}^{2} x_{5,4} x_{5,5}$.\\

We will again define an analogous multi-variate polynomial $P_{n}$ on all the previously defined variables. However, for notational convenience, the $x_{i,j}$ variables will be written as a matrix of variables:
\begin{align}
	X_{n} := \left[ \begin{array}{ccccc} 
		x_{1,1} & & \cdots & & x_{1,n}\\
		 & \ddots & & & \\
		\vdots & & x_{i,i} & & \vdots \\
		 & & & \ddots & \\
		x_{n,1} & & \cdots & & x_{n,n}
		\end{array} \right]
\end{align}
where we will disregard the entries above the diagonal (i.e., the $x_{i,j}$ entries where $j > i$).\\

For each $n$, we now define the polynomial
\begin{align*}
	P_{n}(t; X_{n}) := \mathop{\sum} \limits_{\pi \in \mathcal{S}_{n}} \text{weight}(\pi).
\end{align*}
Recall that $P_{n}(t; \boldsymbol{1}) = f_{n}(t)$, where $\boldsymbol{1}$ is the matrix of all $1$'s. We will derive a functional equation for this $P_{n}$ function, but first, we derive the following lemma:\\

\begin{lemma}
	Let $\pi = \pi_{1} \ldots \pi_{n}$ and suppose that $\pi_{1} = i$. If $\pi^{\prime} := \red(\pi_{2} \ldots \pi_{n})$, then
	\begin{align*}
		\weight(\pi) = x_{i,1}^{0} x_{i,2}^{1} \ldots x_{i,i}^{i-1} \cdot \weight(\pi^{\prime}) \vert_{A} \quad ,
	\end{align*}
	where $A$ is the set of substitutions given by
	\begin{align*}
		A := \begin{cases}
			x_{b,c} \rightarrow x_{b+1,c} & b \geq i, c < i\\
			x_{b,c} \rightarrow x_{b+1,c+1} & b \geq i, c > i\\
			x_{b,c} \rightarrow t x_{b+1,c} \cdot x_{b+1,c+1} & b \geq i, c = i \quad .
			\end{cases}
	\end{align*}
\end{lemma}

\begin{proof}
	We assume $i$ to be a fixed value. Observe that $N_{2 3 1}(\pi)$ is equal to the number of occurrences of $2 3 1$ in $\pi_{2} \ldots \pi_{n}$ \textbf{plus} the number of occurrences of $2 1$ in $\pi_{2} \ldots \pi_{n}$, where the term corresponding to the ``$2$'' is greater than $i$ and the term corresponding to the ``$1$'' is less than $i$. We make the following two observations. First, in $\weight(\pi)$, the exponents of $x_{k,i}$ and $x_{k,i+1}$ are equal for each $k$ (since $\pi_{1} = i$). Second, the number of $2 3 1$ patterns that include the first term $\pi_{1} = i$ is the sum of the exponents of $x_{j,i}$ for $i+1 \leq j \leq n$.
	
	If we re-insert $i$ at the beginning of $\pi^{\prime}$, we would shift all the terms $i, i+1, \ldots , n-1$ up by $1$. This (combined with the prior observations) would lead to the set of substitutions given by $A$. Note that there is no case for $b < i, c \geq i$ since the $x_{b,c}$ variables are only defined for $b \geq c$. Finally, the new ``$i$'' would create new $2 1$ patterns and would require an extra factor of $x_{i,1}^{0} x_{i,2}^{1} \ldots x_{i,i}^{i-1}$ for the weight. 
\end{proof}


Now, define the operator $R_{1}$ on an $n \times n$ square matrix $X_{n}$ and $i < n$ to be:
\begin{align}
	R_{1}(X_{n},i) := \left[ \begin{array}{ccccccc} 
		x_{1,1} & \cdots & x_{1,i-1} & t x_{1,i} x_{1,i+1} & x_{1,i+2} & \cdots & x_{1,n} \\
		\vdots & \ddots & & \vdots & & & \vdots \\
		x_{i-1,1} &  & x_{i-1,i-1} &  & \cdots & & x_{i-1,n} \\
		x_{i+1,1} & \cdots & x_{i+1,i-1} & t x_{i+1,i} x_{i+1,i+1} & x_{i+1,i+2} & \cdots & x_{i+1,n} \\
		\vdots & & \vdots & \vdots & \ddots & & \vdots\\
		\vdots & & \vdots & \vdots & & \ddots & \vdots\\
		x_{n,1} & \cdots & x_{n,i-1} & t x_{n,i} x_{n,i+1} & x_{n,i+2} & \cdots & x_{n,n}
		\end{array} \right]. \label{R1}
\end{align}
In essence, the $R_{1}$ operator deletes the $i$-th row, merges the $i$-th and $(i+1)$-th columns via term-by-term multiplication, and multiplies this new column by a factor of $t$. If $i = n$, then $R_{1}(X_{n},i)$ is defined to be the $(n-1) \times (n-1)$ matrix obtained by deleting the $n$-th row and $n$-th column from $X_{n}$. It is important to note that while this operator is defined on any $n \times n$ matrix, it will only be applied to our ``matrix of variables'' $X_{n}$ to get a smaller $(n-1) \times (n-1)$ matrix.\\

The previous lemma now leads directly to the following:\\

\begin{thm}
	For the pattern $\tau = 2 3 1$, 
	\begin{align}
		P_{n}(t; X_{n}) = \mathop{\sum} \limits_{i=1}^{n} {x_{i,1}^{0} x_{i,2}^{1} \ldots x_{i,i}^{i-1} \cdot P_{n-1}(t; R_{1}(X_{n},i))}. \tag{FE231} \label{FE231}
	\end{align}
\end{thm}
\noindent Note that while all entries in the matrix are changed for consistency, we will continue to disregard the entries above the diagonal.\\

Again, our goal is to compute $P_{n}(t; \boldsymbol{1})$, and the analogous computational techniques from previous sections will also apply in this setting. For example, we can apply (\ref{FE231}) directly to $P_{n}(t; \boldsymbol{1})$ as opposed to computing $P_{n}(t; X_{n})$ symbolically and substituting $x_{i,j} = 1$ at the end. The following result, which is obvious from the definition of the operator $R_{1}$, provides a substantial simplification:
\begin{lemma}
	\label{sqred}
	Let $A$ be a square matrix where every row is identical (i.e., the $i$-th row and the $j$-th row are equal for every $i,j$). Then, $R_{1}(A,i)$ will also be a square matrix with identical rows.
\end{lemma}

By Lemma~\ref{sqred}, repeated applications of $R_{1}$ to the all ones matrix $\boldsymbol{1}$ will still result in a matrix with identical rows. Therefore, it is sufficient to keep track of only one row as opposed to the entire matrix. Also observe that repeated applications of $R_{1}$ to the matrix $\boldsymbol{1}$ will always result in a matrix whose entries are powers of $t$. Let $Q_{n}(t; c_{1}, \ldots, c_{n})$ denote the polynomial $P_{n}(t; C)$, where $C$ is the $n \times n$ matrix where every row is $\left[ c_{1}, \ldots, c_{n} \right]$ and every $c_{i}$ is a power of $t$. This leads to a functional equation analogous to (\ref{FE231}):
\begin{align}
	Q_{n}(t; c_{1}, \ldots, c_{n}) = \mathop{\sum} \limits_{i=1}^{n} {c_{1}^{0} c_{2}^{1} \ldots c_{i}^{i-1} \cdot Q_{n-1}(t; c_{1}, \ldots, c_{i-1}, t c_{i} c_{i+1}, c_{i+2}, \ldots, c_{n})}. \tag{FE231c} \label{FE231c}
\end{align}
Note that $Q_{n}(t; 1 \left[n \text{ times} \right])$ is exactly our desired polynomial $P_{n}(t; \boldsymbol{1}) = f_{n}(t)$. However, this interpretation only forces us to deal with $n$ catalytic variables (the $c_{i}$'s) as opposed to $n(n+1)/2$ catalytic variables (the $x_{i,j}$'s). Just as in prior sections, we can repeatedly apply our functional equation (\ref{FE231c}) to compute $Q_{n}(t; 1 \left[n \text{ times} \right])$.\\

When the sequence $s_{n}(2 3 1, r)$ is desired for a fixed $r$, the obvious analog of Lemma~\ref{crucial} and the computational reduction using the $\text{CHOP}_{r}$ operator can again be used. This has been implemented in the Maple package {\tt F231}.\footnote{Although all output would be equivalent to the $1 3 2$ case, the approach here will be necessary when considering multiple patterns.}

\subsection{Extending to the pattern $2 3 4 1$}
In this section, we outline how to extend the approach for $2 3 1$ to an analogous (but more complicated) approach for $2 3 4 1$. In addition to the variable $t$, we now introduce $n(n+1)/2$ catalytic variables $x_{i,j}$ with $1 \leq j \leq i \leq n$ and $n(n+1)/2$ more catalytic variables $y_{i,j}$ with $1 \leq j \leq i \leq n$ (a total of $n(n+1)$ catalytic variables). Define the weight of a permutation $\pi = \pi_{1} \ldots \pi_{n}$ to be
\begin{gather*}
	\text{weight}(\pi) := \\
	t^{N_{2 3 4 1}(\pi)} \mathop{\prod} \limits _{1 \leq j \leq i \leq n} x_{i,j}^{\# \{ (a,b) \; : \; \pi_{a} > \pi_{b} , \; \pi_{a}=i, \; \pi_{b}<j, \; 1 \leq a < b \leq n\}} \cdot y_{i,j}^{\# \{ (a,b,c) \; : \; \pi_{c} < \pi_{a} < \pi_{b} , \; \pi_{a}=i, \; \pi_{c}<j, \; 1 \leq a < b < c \leq n\}}
\end{gather*}
For example, $\text{weight}(2 4 3 5 1) = t^{2} x_{2,2} x_{3,2} x_{3,3} x_{4,2} x_{4,3} x_{4,4}^{2} x_{5,2} x_{5,3} x_{5,4} x_{5,5} y_{2,2}^{3} y_{3,2} y_{3,3} y_{4,2} y_{4,3} y_{4,4}$.\\

The $x_{i,j}$ variables and the $y_{i,j}$ variables will be written as matrices of variables:
\begin{align}
	X_{n} := \left[ \begin{array}{ccccc} 
		x_{1,1} & & \cdots & & x_{1,n}\\
		 & \ddots & & & \\
		\vdots & & x_{i,i} & & \vdots \\
		 & & & \ddots & \\
		x_{n,1} & & \cdots & & x_{n,n}
		\end{array} \right], \; \; \; \;
	Y_{n} := \left[ \begin{array}{ccccc} 
		y_{1,1} & & \cdots & & y_{1,n}\\
		 & \ddots & & & \\
		\vdots & & y_{i,i} & & \vdots \\
		 & & & \ddots & \\
		y_{n,1} & & \cdots & & y_{n,n}
		\end{array} \right]
\end{align}
where we will disregard the entries above the diagonal.\\

For each $n$, we define the polynomial
\begin{align*}
	P_{n}(t; X_{n}, Y_{n}) := \mathop{\sum} \limits_{\pi \in \mathcal{S}_{n}} \text{weight}(\pi)
\end{align*}
and again $P_{n}(t; \boldsymbol{1}, \boldsymbol{1}) = f_{n}(t)$ is our desired polynomial. We now have the following result:\\

\begin{lemma}
	Let $\pi = \pi_{1} \ldots \pi_{n}$ and suppose that $\pi_{1} = i$. If $\pi^{\prime} := \red(\pi_{2} \ldots \pi_{n})$, then
	\begin{align*}
		\weight(\pi) = x_{i,1}^{0} x_{i,2}^{1} \ldots x_{i,i}^{i-1} \cdot \weight(\pi^{\prime}) \vert_{A^{\prime}} \quad ,
	\end{align*}
	where $A^{\prime}$ is the set of substitutions given by
	\begin{align*}
		A^{\prime} := \begin{cases}
			x_{b,c} \rightarrow y_{i,c} \cdot x_{b+1,c} & b \geq i, c < i\\
			x_{b,c} \rightarrow x_{b+1,c+1} & b \geq i, c > i\\
			x_{b,c} \rightarrow y_{i,i} \cdot x_{b+1,c} \cdot x_{b+1,c+1} & b \geq i, c = i\\
			y_{b,c} \rightarrow y_{b+1,c} & b \geq i, c < i\\
			y_{b,c} \rightarrow y_{b+1,c+1} & b \geq i, c > i\\
			y_{b,c} \rightarrow t y_{b+1,c} \cdot y_{b+1,c+1} & b \geq i, c = i \quad .
		\end{cases}
	\end{align*}
\end{lemma}

\begin{proof}
	We assume $i$ to be a fixed value. Observe that $N_{2 3 4 1}(\pi)$ is equal to the number of occurrences of $2 3 4 1$ in $\pi_{2} \ldots \pi_{n}$ \textbf{plus} the number of occurrences of $2 3 1$ in $\pi_{2} \ldots \pi_{n}$, where the term corresponding to the ``$2$'' is greater than $i$ and the term corresponding to the ``$1$'' is less than $i$. We make the following few observations. First, in $\weight(\pi)$, the exponents of $x_{k,i}$ and $x_{k,i+1}$ are equal and the exponents of $y_{k,i}$ and $y_{k,i+1}$ are equal for each $k$ (since $\pi_{1} = i$). Second, the number of $2 3 4 1$ patterns that include the first term $\pi_{1} = i$ is the sum of the exponents of $y_{j,i}$ for $i+1 \leq j \leq n$. Third, the number of $2 3 1$ patterns that include the first term $\pi_{1} = i$ (i.e., the ``$2$'' is equal to $i$) and whose ``$1$'' term is less than $k$ is equal to the sum of the exponents of $x_{j,k}$ for $i+1 \leq j \leq n$.
	
	If we re-insert $i$ at the beginning of $\pi^{\prime}$, we would shift all the terms $i, i+1, \ldots , n-1$ up by $1$. This (combined with the prior observations) would lead to the set of substitutions given by $A^{\prime}$. Note that there is no case for $b < i, c \geq i$ since the $x_{b,c}$ variables are only defined for $b \geq c$. Finally, the new ``$i$'' would create new $2 1$ patterns and would require an extra factor of $x_{i,1}^{0} x_{i,2}^{1} \ldots x_{i,i}^{i-1}$ for the weight. 
\end{proof}

In addition to the previous $R_{1}$ operator defined in Eq.~\ref{R1}, we define another operator $R_{2}$ on two square matrices $X_{n}$ and $Y_{n}$ (of equal dimension) and $i<n$ to be: 
\begin{align}
	R_{2}(X_{n},Y_{n},i) := \left[ \begin{array}{ccccccc} 
		x_{1,1} & \cdots & x_{1,i-1} & y_{i,i} x_{1,i} x_{1,i+1} & x_{1,i+2} & \cdots & x_{1,n} \\
		\vdots & \ddots & & \vdots & & & \vdots \\
		x_{i-1,1} &  & x_{i-1,i-1} &  & \cdots & & x_{i-1,n} \\
		y_{i,1} x_{i+1,1} & \cdots & y_{i,i-1} x_{i+1,i-1} & y_{i,i} x_{i+1,i} x_{i+1,i+1} & x_{i+1,i+2} & \cdots & x_{i+1,n} \\
		\vdots & & \vdots & \vdots & \ddots & & \vdots\\
		\vdots & & \vdots & \vdots & & \ddots & \vdots\\
		y_{i,1} x_{n,1} & \cdots & y_{i,i-1} x_{n,i-1} & y_{i,i} x_{n,i} x_{n,i+1} & x_{n,i+2} & \cdots & x_{n,n}
		\end{array} \right]. \label{R2}
\end{align}
In essence, the $R_{2}$ operator deletes the $i$-th row, merges the $i$-th and $(i+1)$-th columns via term-by-term multiplication (and multiplies this new column by a factor of $y_{i,i}$), and scales all $x_{b,c}$ with $b > i$ and $c < i$ by terms from $Y_{n}$. If $i=n$, then $R_{2}(X_{n},Y_{n},i)$ is defined to be the $(n-1) \times (n-1)$ matrix obtained by deleting the $n$-th row and $n$-th column from $X_{n}$.\\

The previous lemma now leads to the following:
\begin{thm}
	For the pattern $\tau = 2 3 4 1$, 
	\begin{align}
		P_{n}(t; X_{n},Y_{n}) = \mathop{\sum} \limits_{i=1}^{n} {x_{i,1}^{0} x_{i,2}^{1} \ldots x_{i,i}^{i-1} \cdot P_{n-1}(t; R_{2}(X_{n},Y_{n},i), R_{1}(Y_{n},i))}. \tag{FE2341} \label{FE2341}
	\end{align}
\end{thm}
As in prior sections, we recursively apply the functional equation directly to $P_{n}(t; \boldsymbol{1}, \boldsymbol{1})$ (and subsequent instances of $P_{k}$). Observe that in this scenario, Lemma~\ref{sqred} still applies for the $R_{1}$ operator and more specifically the ``$Y_{n}$'' matrix in $P_{n}$. While the lemma does not apply to the $R_{2}$ operator, this still allows us to reduce the number of catalytic variables. Let $Q_{n}(t; C; d_{1}, \ldots, d_{n})$ denote the polynomial $P_{n}(t; C, D)$ where every entry of the $n \times n$ matrices $C$ and $D$ are powers of $t$ and every row in $D$ is $\left[ d_{1}, \ldots, d_{n} \right]$. We derive an analogous functional equation:
\begin{align}
	Q_{n}(t; C; \; d_{1}, \ldots, d_{n}) = \mathop{\sum} \limits_{i=1}^{n} {c_{i,1}^{0} c_{i,2}^{1} \ldots c_{i,i}^{i-1} \cdot Q_{n-1}(t; R_{2}(C,D,i); \; d_{1}, \ldots, d_{i-1}, t d_{i} d_{i+1}, d_{i+2}, \ldots, d_{n})}. \tag{FE2341c} \label{FE2341c}
\end{align}
Using this recurrence to compute $Q_{n}(t; \boldsymbol{1}; \; 1 \left[n \text{ times} \right])$ will yield the desired polynomial $f_{n}(t)$. This approach allows us to deal with $n(n+1)/2 + n$ catalytic variables (as opposed to $n(n+1)$ such variables).\\

Additionally, for a fixed $r$, the sequence $s_{n}(2 3 4 1, r)$ can be computed by applying Lemma~\ref{crucial} and the $\text{CHOP}_{r}$ operator as necessary. This has been implemented in the procedure {\tt F2341rN(r,N)} (in the Maple package {\tt F2341}).\\

For example, the Maple call {\tt F2341rN(1,15);} for the first $15$ terms of $s_{n}(2 3 4 1,1)$ produces the sequence:
\begin{gather*}
	0, 0, 0, 1, 11, 87, 625, 4378, 30671, 216883, 1552588, 11257405, 82635707, 613600423, 4604595573
\end{gather*}
and the Maple call {\tt F2341rN(2,15);} for the first $15$ terms of $s_{n}(2 3 4 1,2)$ produces the sequence:
\begin{gather*}
	0, 0, 0, 0, 5, 68, 626, 5038, 38541, 289785, 2172387, 16339840, 123650958, 942437531, 7236542705\\
\end{gather*}

While we do not present the details here, the same methodology can be applied to longer patterns of the form $2 3 \ldots k 1$. Analogous functional equations can be derived and used for enumeration.


\section{Further extensions}\label{multipatt}

\subsection{Tracking inversions}
One of the most commonly studied permutation statistic is the inversion number. The inversion number of a permutation $\pi = \pi_{1} \ldots \pi_{n}$, denoted by $\text{inv}(\pi)$, is the number of pairs $(i,j)$ such that $1 \leq i < j \leq n$ and $\pi_{i} > \pi_{j}$. Equivalently, it is the number of occurrences of the pattern $2 1$ in $\pi$. For a (fixed) pattern $\tau$, define the polynomial
\begin{align}
	g_{n}(t,q) := \mathop{\sum} \limits_{\pi \in \mathcal{S}_{n}} q^{\text{inv}(\pi)} t^{N_{\tau}(\pi)}.
\end{align}
Observe that $g_{n}(t,1)$ is exactly $f_{n}(t)$ from before.\\

Given a permutation $\pi = \pi_{1} \ldots \pi_{n}$, suppose that $\pi_{1} = i$. Then, $\text{inv}(\pi)$ is equal to the number of inversions in $\pi_{2} \ldots \pi_{n}$ plus the number of elements in $\pi_{2}, \ldots, \pi_{n}$ that are less than $i$. For any previously defined functional equation, it is enough to insert a $q^{i-1}$ factor in the summation.\\

For example, if the fixed pattern is $\tau = 1 3 2$, the polynomial $P_{n}$ can be analogously defined as
\begin{align*}
	P_{n}(t,q; x_{1}, \ldots, x_{n}) := \mathop{\sum} \limits_{\pi \in \mathcal{S}_{n}} q^{\text{inv}(\pi)} \cdot \text{weight}_{132}(\pi)
\end{align*}
and the analog to functional equation (\ref{FE132}) would be
\begin{align*}
	P_{n}(t,q; x_{1}, \ldots, x_{n}) = \mathop{\sum} \limits_{i=1}^{n} {q^{i-1} x_{1} x_{2} \ldots x_{i-1} \cdot P_{n-1}(t,q; x_{1}, \ldots, x_{i-1}, t x_{i+1}, \ldots, t x_{n})}. 
\end{align*}
Similarly, the analogous functional equation to (\ref{FE231c}) would be
\begin{align*}
	Q_{n}(t,q; c_{1}, \ldots, c_{n}) = \mathop{\sum} \limits_{i=1}^{n} {q^{i-1} c_{1}^{0} c_{2}^{1} \ldots c_{i}^{i-1} \cdot Q_{n-1}(t,q; c_{1}, \ldots, c_{i-1}, t c_{i} c_{i+1}, c_{i+2}, \ldots, c_{n})}. 
\end{align*}
From here, all the previous computational techniques for quick enumeration still apply.\\

This has been implemented in the procedures {\tt qF123r(n,r,t,q)} and {\tt qF1234r(n,r,t,q)} (in Maple package {\tt FINCR}), {\tt qF132r(n,r,t,q)} and {\tt qF1243r(n,r,t,q)} (in Maple package {\tt FINCRT}),  {\tt qF231r(n,r,t,q)} (in Maple package {\tt F231}), and {\tt qF2341r(n,r,t,q)} (in Maple package {\tt F2341}).

\subsection{Counting multiple patterns in permutations}
In the preceding sections, various functional equations were derived by considering the first term of a typical permutation and deriving a recurrence. Given any collection of patterns where such recurrences can be derived, we can also consider those patterns simultaneously.\\

As an example, consider the patterns $\sigma = 1 2 3$ and $\tau = 1 3 2$. The case of only the pattern $1 2 3$ was done in~\cite{nz:gwilf}. In this setting, the weight of a permutation $\pi = \pi_{1} \ldots \pi_{n}$ is defined by
\begin{align*}
	\text{weight}_{123}(\pi) := t^{N_{123}(\pi)} \mathop{\prod} \limits_{i=1}^{n} {x_{i}^{\# \{ (a,b) \; : \; \pi_{a} = i < \pi_{b},\; 1 \leq a < b \leq n\}}}
\end{align*}
and the corresponding polynomial is
\begin{align*}
	P_{n}(t; x_{1}, \ldots, x_{n}) := \mathop{\sum} \limits_{\pi \in \mathcal{S}_{n}} \text{weight}_{123}(\pi).
\end{align*}
The corresponding functional equation (referred to as the \emph{Noonan-Zeilberger Functional Equation}) is
\begin{align}
	P_{n}(t; x_{1}, \ldots, x_{n}) = \mathop{\sum} \limits_{i=1}^{n} x_{i}^{n-i} P_{n-1}(t; x_{1}, \ldots, x_{i-1}, t x_{i+1}, \ldots, t x_{n}). \tag{NZFE} \label{NZFE}
\end{align}
This can be merged with the analogous quantities for $1 3 2$ as follows. Let $s$ and $t$ be the variables corresponding to $1 2 3$ and $1 3 2$, respectively. Let $x_{1}, \ldots, x_{n}$ and $y_{1}, \ldots, y_{n}$ be two sets of catalytic variables, and define the weight of a permutation $\pi = \pi_{1} \ldots \pi_{n}$ by
\begin{align*}
	\text{weight}(\pi) := s^{N_{123}(\pi)} t^{N_{132}(\pi)} \mathop{\prod} \limits_{i=1}^{n} {x_{i}^{\# \{ (a,b) \; : \; \pi_{a} = i < \pi_{b},\; 1 \leq a < b \leq n\}}} {y_{i}^{\# \{ (a,b) \; : \; \pi_{a} > \pi_{b} = i,\; 1 \leq a < b \leq n\}}}.
\end{align*}
For each $n$, we define the polynomial
\begin{align*}
	P_{n}(s,t; x_{1}, \ldots, x_{n}, y_{1}, \ldots, y_{n}) := \mathop{\sum} \limits_{\pi \in \mathcal{S}_{n}} \text{weight}(\pi)
\end{align*}
and can similarly derive the functional equation
\begin{gather*}
	P_{n}(s,t; x_{1}, \ldots, x_{n}, y_{1}, \ldots, y_{n})  = \\
	\mathop{\sum} \limits_{i=1}^{n} x_{i}^{n-i} y_{1} y_{2} \ldots y_{i-1} P_{n-1}(s,t; x_{1}, \ldots, x_{i-1}, s x_{i+1}, \ldots, s x_{n}, y_{1}, \ldots, y_{i-1}, t y_{i+1}, \ldots, t y_{n}). 
\end{gather*}
The same computational techniques from the prior sections apply here as well.\\

This has been implemented in the Maple package {\tt F123n132}. For example, the Maple call {\tt F123r132sN(2,2,15);} gives the first $15$ terms of the sequence enumerating permutations with $2$ occurrences of $1 2 3$ and $2$ occurrences of $1 3 2$:
\begin{gather*}
	0, 0, 0, 1, 6, 26, 94, 306, 934, 2732, 7752, 21488, 58432, 156288, 411904
\end{gather*}
and the Maple call {\tt F123r132sN(4,2,15);} gives the first $15$ terms of the sequence enumerating permutations with $4$ occurrences of $1 2 3$ and $2$ occurrences of $1 3 2$:
\begin{gather*}
	0, 0, 0, 0, 1, 5, 23, 106, 450, 1740, 6214, 20831, 66427, 203550, 603920
\end{gather*}

Other pairs (or larger sets) of patterns follow similarly, and the analogous $1 2 3 4$ and $1 2 4 3$ case has been implemented in the Maple package {\tt F1234n1243}.\\

Finally, it is possible to consider all length $3$ patterns simultaneously. Only the patterns $1 2 3$, $1 3 2$, and $2 3 1$ were done directly, but analogous functional equations can be derived for $3 2 1$, $3 1 2$, and $2 1 3$. These six functional equations can be combined to count occurrences of all the length $3$ patterns. This has been implemented in the Maple package {\tt FS3}. For example, the Maple call {\tt FS3full(7,[t[1],t[2],t[3],t[4],t[5],t[6]]);} would produce the polynomial
\begin{gather*}
	\mathop{\sum} \limits_{\pi \in \mathcal{S}_{7}} {t_{1}^{N_{123}(\pi)} t_{2}^{N_{132}(\pi)} t_{3}^{N_{213}(\pi)} t_{4}^{N_{231}(\pi)} t_{5}^{N_{312}(\pi)} t_{6}^{N_{321}(\pi)}}
\end{gather*}
in its computed and expanded form.\footnote{The actual output from Maple is too large to include here. We were able to compute up to {\tt FS3full(11,[t[1],t[2],t[3],t[4],t[5],t[6]]);}, which is a $450$ megabyte text file.}


\section{Conclusion}\label{concl}
In this work, we extended and generalized the techniques of~\cite{nz:gwilf} to the pattern families $1 2 \ldots (k-2) (k) (k-1)$ and $2 3 \ldots k 1$. In addition, we showed how this approach could be further extended to handle inversions and more generally, multiple patterns simultaneously. It would be interesting to see what additional patterns this approach can be applied to. Also, while the main results of this paper are enumeration algorithms, they are based off of rigorously derived functional equations. It would be interesting to find out if any additional information can be extracted from these functional equations.\\

Finally, the techniques of this paper allow us to compute many sequences that lead to new conjectures. Let $c_{r,s}(n)$ be the number of length $n$ permutations with exactly $r$ occurrences of $1 2 3 4$ and exactly $s$ occurrences of $1 2 4 3$. We will denote this as $c(n)$ when $r$ and $s$ is clear from context. When $r = s = 0$, it is known that $c(n)$ is exactly the Schr\"{o}der numbers. For fixed $r,s \leq 1$, we are almost certain that $c(n)$ is P-recursive\footnote{This is a special case of the Noonan-Zeilberger Conjecture from~\cite{noonzeil:forbid}.} (the algorithm can compute enough terms to guess a recurrence). On the other hand, for the single pattern case, it is not clear if $s_{n}(1 2 3 4,1)$ is P-recursive~\cite{nz:gwilf}.\\

Based off of empirical evidence, we also believe the following to be true:
\begin{conj}
	Given fixed $r \geq 0$ and $s \geq 0$, let $a(n)$ be the number of length $n$ permutations with exactly $r$ occurrences of $1 2 3$ and $s$ occurrences of $1 3 2$. Then, there exists a polynomial $p(n)$ of degree $r + s$ such that $a(n) = p(n) 2^{n}$ for all $n \geq r + s + 1$.
\end{conj}
\noindent This conjecture has been empirically checked for all $r+s \leq 10$. There are some results considering this type of problem, but most such results limit themselves to $s = 0,1$. For example, Robertson~\cite{robertson} derives closed form expressions for $a(n)$ when $r = 0,1$ and $s = 0,1$, and other authors~\cite{RWZ, manvain2} derive more general generating functions for $a(n)$ when $s = 0$. If this general form were shown to hold for arbitrary $r$ and $s$, the {\tt F123n132} package could quickly compute enough terms to find explicit formulas and generating functions.\\

\noindent \textbf{Acknowledgments}: The author owes many thanks to Doron Zeilberger for the helpful discussions on this topic. The author also needs to thank the anonymous referees for their corrections and comments toward improving this article.

\bibliography{GWilf2}{}
\bibliographystyle{plain}

\end{document}